\documentclass{amsart}
\usepackage{graphicx,amssymb,amsmath}
\usepackage{enumitem}

\newcommand{\Hc}{\mathcal{H}}

\newcommand{\Zb}{\mathbb{Z}}
\newcommand{\Red}{\mathbb{R}^d}
\newcommand{\C}{\mathcal{C}}
\newcommand{\Ci}{\mathcal{C}}
\newcommand{\vol}[1]{\mathrm{vol}\left(#1\right)}

\newcommand{\cvx}[1][d]{\mathrm{Cvx}(#1)}
\newcommand{\st}{\;:\;}

\newcommand{\Kd}{\mathcal{K}_d}
\newcommand{\Qd}{\mathcal{Q}_d} 
\newcommand{\Quanthyp}{\Qd\left(v^\ell\right)} 
\newcommand{\Quantchain}{\left(\Quanthyp\right)_{\ell \in \mathbb{Z}}} 
\newcommand{\Hch}{(\Hc_\ell)_{\ell \in \mathbb{Z}}} 
\newcommand{\colortimes}{\otimes}

\newtheorem{theorem}{Theorem}
\newtheorem{cor}[theorem]{Corollary}
\newtheorem{lemma}[theorem]{Lemma}
\newtheorem{conj}{Conjecture}
\theoremstyle{definition}
\newtheorem{definition}[theorem]{Definition}

\title{Quantitative Helly-type Theorems via Hypergraph Chains}

\author{Attila Jung}

\address{Department of Computer Science, Eötvös Loránd University, Budapest.}

\thanks{Research was supported by the Rényi Doctoral Fellowship of the Rényi Institute, the NKFIH grant FK132060 and the Thematic Excellence Program TKP2021-NKTA-62 of the National Research, Development and Innovation Office}

\email{jungattila@gmail.com}

\begin{document}
	\thispagestyle{empty}
	\begin{abstract}
		We propose a combinatorial framework to analyze quantitative Helly-type questions. Using this framework, we prove a Quantitative Fractional Helly Theorem with Fractional Helly Number $3d$ and a stability version of the Quantitative Helly Theorem of Bárány, Katchalski, and Pach.
	\end{abstract}
	\maketitle
	
	\section{Introduction}
	Two directions in the study of Helly-type Theorems are \emph{quantitative} and \emph{abstract} questions.  \emph{Quantitative} results concern intersection patterns of convex sets in some specific space, originally $\Red$, where instead of finding points in the intersection, one bounds the size, for example the volume or the diameter of the intersection. \emph{Abstract} results, on the other hand, study more general structures, e.g. hypergraphs, with certain properties that capture some essential aspects of the behavior of convex sets. In this note, we connect the two.
	
	There are quantitative results, where the usual combinatorial techniques are not directly applicable, since more than one intersection pattern of convex sets are involved in them (e.g. convex sets intersecting in \emph{large} and in \emph{small} volumes). In this note, we present a combinatorial framework in which these quantitative Helly-type questions can be analyzed. In particular, we propose the definition of hypergraph chains (see Definitions \ref{def:Hchain}, \ref{def:QHellyChain}, \ref{def:QCHChain} and \ref{def:QFHChain}) and prove our main result, Theorem~\ref{thm:Qholmsen}, that a certain type of Quantitative Colorful Helly Theorem implies a Quantitative Fractional Helly Theorem.
	
	First, consider the \emph{Quantitative Volume Theorem}.
	\begin{theorem}[B\'ar\'any, Katchalski and Pach \cite{BKP82}]\label{thm:QH}
		Assume that the intersection of any $2d$ members of a finite family of convex sets in $\Red$ is of volume at least one. Then the volume of the intersection of all members of the family is of volume at least $c(d)$, a constant depending on $d$ only.
	\end{theorem} 
	In \cite{BKP82}, it is proved that one can take $c(d)=d^{-2d^2}$ and conjectured that it should 
	hold with $c(d)=d^{-cd}$ for an absolute constant $c>0$. Theorem~\ref{thm:QH} was confirmed  with 
	$c(d)\approx d^{-2d}$ by Naszódi \cite{MR3439267}, whose argument was refined by 
	Brazitikos \cite{Bra17}, who showed that one may take $c(d)\approx d^{-3d/2}$. 
	For more on quantitative Helly-type results, see the surveys 
	\cite{DGFM19survey, HW18survey}.
	
	Helly's theorem may be stated in the language of hypergraphs as follows. Let $V$ be a finite family of convex sets in $\Red$, and call a subset of $V$ an edge of our hypergraph, if the intersection of the corresponding convex sets is not empty. Helly's theorem states that if all $(d+1)$-tuples of a subset $S$ of $V$ are edges of the hypergraph, then so is $S$. Observe that Theorem~\ref{thm:QH} cannot be translated to the same language, as two hypergaphs are involved: in one, the edges correspond to families of convex sets whose intersection is of volume at least one, and in the other, this volume is at least $c(d)$. The goal of this note is to provide a combinatorial framework in which Theorem~\ref{thm:QH}, and other quantitative results can be translated.
	
	The \emph{Colorful Helly Theorem}\label{page:chelly} found by 
	Lov\'asz \cite{Lo74} (and with the first published proof by B\'ar\'any 
	\cite{MR676720}) states the following.
	\emph{If 
		$\Ci_1,\dots, \Ci_{d+1}$ are finite families (color classes) of convex sets in 
		$\Red$, such 
		that for any colorful selection $C_1\in \Ci_1,\dots, C_{d+1}\in \Ci_{d+1}$, the 
		intersection $\bigcap\limits_{i=1}^{d+1} C_i$ is non-empty, then for some $j$, 
		the intersection $\bigcap_{C\in \Ci_j} C$ is also non-empty.}
	
	In \cite{damasdi2021colorful}, the following quantitative variant is shown.
	
	\begin{theorem}[Damásdi, Földvári and Naszódi \cite{damasdi2021colorful}]\label{thm:QCH}
		Let $\C_1,\linebreak[0]\ldots,\C_{3d}$ be finite families of convex sets in 
		$\Red$. Assume that for any colorful selection $C_1\in \Ci_1,\dots, C_{3d}\in \Ci_{3d}$, the 
		intersection $\bigcap\limits_{i=1}^{3d} C_i$ is of volume at least one. 
		
		\noindent Then, there is a $j$ with $1\leq j \leq 3d$ such that 
		$\vol{\bigcap\limits_{C\in \C_j}C}\geq d^{-cd^2}$ with a universal constant $c>0$.
	\end{theorem}
	
	The \emph{Fractional Helly Theorem}  
	due to Katchalski and Liu \cite{KL79} (see also \cite[Chapter 8]{MR1899299})
	is another classical Helly-type result, which states the following. 
	\label{page:fracHelly}
	\emph{Fix a dimension $d$, and an $\alpha\in(0,1)$, and let $\C$ be a finite 
		family of convex sets in $\Red$ with the property that among the subfamilies of 
		$\C$ of size $d+1$, there are at least $\alpha \binom{|\C|}{d+1}$ for whom the 
		intersection of the $d+1$ members is nonempty. Then, there is a subfamily 
		$\C^\prime\subset\C$ of size $|\C^{\prime}|\geq\frac{\alpha}{d+1} |\C|$ such 
		that the intersection of all members of $\C^\prime$ is nonempty.}
	
	In \cite{jung2022quantitative}, the following quantitative variant of the Fractional Helly 
	Theorem is shown.
	
	\begin{theorem}[Jung and Naszódi \cite{jung2022quantitative}]\label{thm:QFH}
		For every dimension $d \geq 1$ and every $\alpha\in(0,1)$, there is a 
		$\beta\in(0,1)$ such that the following holds.
		
		\noindent Let $\C$ be a finite family of convex sets in $\Red$. Assume that among all 
		subfamilies of size $3d+1$, there are at least $\alpha \binom{|\C|}{3d+1}$ for 
		whom the intersection of the $3d+1$ members is of volume at least one.
		
		\noindent Then, there is a subfamily $\C^\prime\subset\C$ of size at least 
		$\beta |\C|$ such that $\vol{\bigcap\limits_{C\in \C^\prime}C}\geq d^{-cd^2}$ with a universal constant $c>0$.
	\end{theorem}
	
	In Theorems~\ref{thm:QH}, \ref{thm:QCH} and \ref{thm:QFH}, the cardinalities $2d$, $3d$ and $3d+1$ appear, respectively. It is easy to verify (cf. \cite{BKP82}) that Theorem~\ref{thm:QH} does not hold with any number below $2d$, which implies the same lower bound for the other two theorems. No better lower bounds are known. 
	
	Turning to abstract results, we describe Helly's Theorem and the Fractional and  Colorful Helly Theorems in the language of hypergraphs.
	Let $V$ be a (possibly infinite) set. A \emph{hypergraph} on the base set $V$ is 
	any family of its subsets, $\Hc \subset 2^V$. A hypergraph is downwards closed, if $H \in \Hc$ and $G \subset H$ implies $G \in \Hc$. A downwards closed hypergraph $\Hc$ has \emph{Helly Number} $h$, if for every finite subset $S \subset V$ the relation $\binom{S}{h} \subset \Hc$ implies $S \in \Hc$. Now let us denote the family of convex sets of $\Red$ as $\cvx$ and the hypergraph which contains the subfamilies of convex sets with nonempty intersection by $\Kd = \{\C \subset \cvx: \cap_{C\in \C}C \neq \emptyset\}$. Helly's Theorem says that $\Kd$ has Helly-number $d+1$.
	
	A downwards closed hypergraph $\Hc$ over a base set $V$ has \emph{Fractional Helly Number} $k$, if there exists a function $\beta: (0,1) \to (0,1)$ such that whenever $S \subset V$ is a finite subset such that $\left |\Hc \cap \binom{S}{k} \right |$, the number of edges of $\Hc$ of size $k$ in $S$ is at least $\alpha \binom{|S|}{k}$ with an $\alpha \in (0,1)$, then there exists a subset $S^\prime \subset S$ of size at least $\beta |S|$ such that $S^\prime \in \Hc$. The Fractional Helly Theorem says, that $\Kd$ has Fractional Helly Number $d+1$.
	
	We turn to phrasing the  Colorful Helly Theorem in an abstract setting.
	Let $S_1, \ldots, S_k\subset V$ be (not necessarily disjoint) subsets of a base set $V$, which we will call color classes.
	We call a set $F\subset V$ a colorful selection from these color classes, if
	$F$ contains one element from each color class. Very formally, to clarify how elements belonging to multiple color classes are handled, we say that $F\subset V$ is a \emph{colorful selection}, if there is a surjective map $\phi:[k]\longrightarrow F$ with $\phi(i)\in S_i$ for all $i\in [k]$. We denote the set of colorful selections by
	$S_1\colortimes\dots\colortimes S_k$.
	
	A downwards closed hypergraph $\Hc$ over a base set $V$ has \emph{Colorful Helly Number} $k$, if for every $k$ finite subset $S_1, \ldots, S_k \subset V$ such that $(S_1\colortimes\dots\colortimes S_k)\subseteq \Hc$, there exists a color class $S_j$ with $S_j \in \Hc$. The Colorful Helly Theorem says that $\Kd$ has  Colorful Helly Number $d+1$.
	
	Alon, Kalai, Matou{\v s}ek and Meshulam \cite{Alon02} considered Helly-type 
	results in the abstract setting. They showed, that if a hypergraph has bounded Fractional Helly Number, then it also has the 
	so called $(p,q)$ property (see the definition in \cite{Alon02}). Holmsen 
	\cite{holmsen2020large} showed that if a hypergraph has Colorful Helly Number $k$, then it has Fractional Helly Number at most $k$. In this sense, the Fractional Helly Theorem can be deduced from the Colorful Helly Theorem with a purely combinatorial proof. Note that Holmsen's result does not immediately imply a similar relationship between Theorem~\ref{thm:QCH} and Theorem~\ref{thm:QFH}, because there are two different kinds of intersection of convex sets (sets intersecting in volume one and sets intersecting in volume $d^{-cd^2}$). 
	
	In Section~\ref{sec:chains}, we introduce the notion of hypergraph chains, and our main results, Theorems~\ref{thm:Qholmsen} and \ref{thm:stabHelly}, which state that Holmsen's results extend to hypergraph chains. As a result, they can be applied in the context of quantitative Helly-type questions. The proof of Theorems~\ref{thm:Qholmsen} and \ref{thm:stabHelly} are contained in Section~\ref{sec:proofs}. In Section~\ref{sec:geometric} we show geometric consequences of our main results. Our Theorem~\ref{thm:Qholmsen} implies (see Corollary~\ref{cor:QFH3d}) that in Theorem~\ref{thm:QFH} we can decrease the number $3d+1$ to $3d$ (at the expense of a bigger loss of volume). Our second main result, Theorem~\ref{thm:stabHelly} implies that Theorem~\ref{thm:QH} is stable: one does not need to check that all $2d$-tuples of the given convex sets have intersection of volume at least one. Instead, it is sufficient to verify it for almost all of them to obtain that almost all have an intersection of some positive volume. Finally, in Section~\ref{sec:rem} we state open questions.
	
	Quantitative Helly-type theorems are considered in \cite{MR3602856,RoSo17,sarkar2021quantitative} with the focus on convex sets in $\Red$, or the lattice $\mathbb{Z}^d$, 
	or sets in topological spaces with particular topological properties. To our 
	knowledge, ours is the first attempt to address quantitative Helly-type questions in the general 
	context of hypergraphs in the spirit of the results of \cite{Alon02} and 
	\cite{holmsen2020large}.
	
	\section{Hypergraph Chains}\label{sec:chains}
	
	\begin{definition}\label{def:Hchain}
		Let $V$ be a (possibly infinite) set. The infinite sequence $\Hch$ of hypergraphs over the base set $V$ is a \emph{hypergraph chain}, if every $\Hc_\ell$ is downwards closed and for all $\ell \in \Zb$, $\Hc_\ell \subset \Hc_{\ell + 1}$.
	\end{definition}
	
	If $V = \cvx$ and $\Hc_\ell = \Kd$ for all $\ell$, then $\Hch$ is a hypergraph chain. A more interesting example is when $V = \cvx$, $v \in (0,1)$ a real number and for an $\ell \in \Zb$, a family of convex sets from $\Red$ is an edge in $\Hc_\ell$, if and only if their intersection is of volume at least $v^\ell$. We will denote this hypergraph by $\Quanthyp$.
	
	\begin{definition}\label{def:QHellyChain}
		A hypergraph chain $\Hch$ over a base set $V$ has \emph{Helly Number} $h$, if for every $S \subseteq V$, $\binom{S}{h} \subset \Hc_\ell$ implies $S \in \Hc_{\ell + 1}$.
	\end{definition}
	
	According to this definition, $(\Kd)_{\ell \in \mathbb{Z}}$ has Helly Number $d+1$.
	
	More interestingly, Theorem~\ref{thm:QH} states that if $v \approx d^{-3d/2}$, then $\Quantchain$ has Helly Number $2d$.

	\begin{definition}\label{def:QCHChain}
		A hypergraph chain $\Hch$ over a base set $V$ has \emph{Colorful Helly Number} $k$, if whenever $S_1, \ldots, S_k$ are finite subsets (color classes) of $V$ and $S_1 \colortimes \ldots \colortimes S_k \subset \Hc_\ell$, then there is a color class $S_j$ with $ S_j \in \Hc_{\ell + 1}$.
	\end{definition}
	
	Note that by taking $S_1 = S_2 = ... = S_k = S$, a hypergraph chain with Colorful Helly Number $k$ has Helly Number $h \leq k$.
	
	According to the definition, $(\Kd)_{\ell \in \mathbb{Z}}$ has Colorful Helly Number $d+1$.
	
	More interestingly, the Quantitative Colorful Helly Theorem, 
	Theorem~\ref{thm:QCH} may be stated as follows. If $v = d^{-cd^2}$ from Theorem~\ref{thm:QCH}, then $\Quantchain$ has Colorful Helly Number $3d$.
	
	\begin{definition}\label{def:QFHChain}
		A hypergraph chain $\Hch$ over a base set $V$ has \emph{Fractional Helly Number} $k$, if there exists a function $\beta : (0,1) \to (0,1)$ such that for every finite set $S \subset V$, if $|\Hc_\ell \cap \binom{S}{k}| \geq \alpha \binom{|S|}{k}$ with some $\alpha \in (0,1)$, then there exists an $S' \subset S$ with $|S'| \geq \beta (\alpha) |S|$ and $S' \in \Hc_{\ell + 1}$.
	\end{definition}
	
	As in the previous two cases, $(\Kd)_{\ell \in \mathbb{Z}}$ has Fractional Helly Number $d+1$ and Theorem~\ref{thm:QFH} states, that if $v = d^{-cd^2}$ from Theorem~\ref{thm:QFH}, then $\Quantchain$ has Fractional Helly Number $3d+1$.
	
	Now we are ready to state our main result, which is a quantitative analogue of Theorem 3 from \cite{holmsen2020large}.
	
	\begin{theorem}\label{thm:Qholmsen}
		If the hypergraph chain $\Hch$ has Colorful Helly Number $k$, then $(\Hc_{(k+1)\ell})_{\ell \in \Zb}$ has Fractional Helly Number $k$.
	\end{theorem}
	
	Here, the obtained Fractional Helly Number is the same as the assumed Colorful Helly Number, but not for the exact same hypergraph chain: we can only take every $(k+1)$th element from the original chain. Can the Fractional Helly number go below the Colorful Helly number? If for a hypergraph chain the Helly Number is smaller than the Colorful Helly Number, the answer is a partial yes.
	
	\begin{theorem}\label{thm:stabHelly}
		If the hypergraph chain $\Hch$ has Helly Number $h$ and Colorful Helly Number $k\geq h$, then there exists a function $\beta: (0,1) \to [0,1)$ with $\lim_{\alpha \to 1} \beta ( \alpha) = 1 $ such that for every finite set $S \subset V$, if $|\Hc_\ell \cap \binom{S}{h}| \geq \alpha \binom{|S|}{h}$ with some $\alpha \in (0,1)$, then there exists an $S' \subset S$ with $|S'| \geq \beta (\alpha) |S|$ and $S' \in \Hc_{\ell+3}$.
	\end{theorem}
	We can interpret this result as a stability version of the Helly property (under some additional assumptions), since $\lim_{\alpha \to 1} \beta ( \alpha) = 1$.
	
	As far as we know, the best possible $\beta$ here might assign $0$ to a large fraction of $\alpha$s from $(0,1)$, this is the difference from hypergraph chains with Fractional Helly Number $h$, where this is not possible. But at least, if $\alpha$ is very close to $1$, then $\beta(\alpha)$ is also close to $1$.

	\section{Proof of Theorems~\ref{thm:Qholmsen} and \ref{thm:stabHelly}}\label{sec:proofs}
	
	Let us begin with an analogue of Lemma 3.1 from \cite{holmsen2020large}. We denote by $\omega_h(\Hc_\ell|_S)$ the size of the largest $h$-clique of $S$, ie. the size of the largest subset $K \subset S$ such that $\binom{K}{h} \subset \Hc_\ell$.
	
	\begin{lemma}\label{lem:kim}
		Let $\Hch$ be a hypergraph chain with Helly Number $h$ and Colorful Helly Number $k$ over a base set $V$. Then for every finite subset $S \subset V$, we have
		\begin{enumerate}[label=(\alph*)]
			\item\label{item:heqk} 
			\hspace{.02\textwidth}$\left |\binom{S}{k} \setminus \Hc_\ell \right | \geq \binom{\frac{1}{k}(|S| - \omega_k (\Hc_{\ell + 1}|_S))}{k}$, and
			\item\label{item:hlessk} 
			\hspace{.02\textwidth}$\left |\binom{S}{h} \setminus \Hc_\ell \right | \geq  \binom{k}{h}^{-1}\binom{\frac{1}{h}(|S| - \omega_h (\Hc_{\ell + 2}|_S))}{h}$.
		\end{enumerate}
	\end{lemma}
	
	\begin{proof}
		Note that $h\leq k$ holds for every hypergraph chain of Helly Number $h$ and Colorful Helly Number $k$. Fix $\ell \in \Zb$. 
		
		For the proof of part \ref{item:heqk}, let $\{M_1, \ldots, M_t\} \subset \binom{S}{k} \setminus \Hc_{\ell + 1}$ be a maximal size family of disjoint missing edges from $\Hc_{\ell + 1}$, each of size $k$. By the maximality of this family, we have $\binom{S \setminus (M_1 \cup \ldots \cup M_t)}{k} \subset \Hc_{\ell + 1}$, and thus, $\omega_k(\Hc_{\ell + 1}|_S) \geq  |S\setminus (M_1 \cup \ldots \cup M_t)| = |S| - tk$ or, equivalently,
		\begin{equation}\label{eq:boundontk}
			t \geq \frac{1}{k}((|S| - \omega_k(\Hc_{\ell + 1}|_S))).	 
		\end{equation}
		
		Consider a selection $I \in \binom{[t]}{k}$ of $k$ indices. Since each $M_i$ is a missing edge from $\Hc_{\ell + 1}$, we have that $\{M_i\st i\in I\}$ is a family of $k$ color classes, such that neither one is contained in $\Hc_{\ell + 1}$. Since $\Hch$ has Colorful Helly Number $k$,
		there is a colorful selection $\{v_i\st i\in I\}\subset V$ of vertices (that is, $v_i\in M_i$ for all $i\in I$) such that $\{v_i\st i\in I\}$ is not an edge in $\Hc_{\ell}$.
		
		Observe that if $I_1, I_2 \in \binom{[t]}{k}$ are distinct selections of indices, then, by the disjointness of the $M_j$, we have that	$\{v_i\st i\in I_1\}\neq \{v_i\st i\in I_2\}$.
		Thus, we found $\binom{t}{k}$ members of $\binom{S}{k} \setminus \Hc_{\ell}$, completing the proof of part \ref{item:heqk}.
		
		For the proof of part \ref{item:hlessk}, let $\{M_1, \ldots, M_t\} \subset \binom{S}{h} \setminus \Hc_{\ell + 2}$ be a maximal size family of disjoint missing edges from $\Hc_{\ell + 2}$, each of size $h$. Similarly to the argument in part \ref{item:heqk}, we have
		\begin{equation}\label{eq:boundonth}
			t \geq \frac{1}{h}((|S| - \omega_h(\Hc_{\ell + 2}|_S))).	 
		\end{equation}
		
		Consider a selection $I \in \binom{[t]}{k}$ of $k$ indices. Again, as in the proof of part \ref{item:heqk}, since $\Hch$ has Colorful Helly Number $k$,
		there is a colorful selection $\{v_i\st i\in I\}\subset V$ of vertices from the color classes $\{M_i\st i\in I\}$ such that $\{v_i\st i\in I\}$ is not an edge in $\Hc_{\ell + 1}$.
		By the Helly property, there is a $J \in \binom{I}{h}$ and an $F \in \binom{S}{h} \setminus \Hc_\ell$ such that $|F \cap M_j| = 1$ for every $j\in J$. Any fixed $J\in\binom{[t]}{h}$ can appear at most $\binom{t-h}{k-h}$ times in this way. Moreover, any fixed $F \in \binom{S}{h} \setminus \Hc_\ell$ may appear for only one $J$, so there are at least $\binom{t}{k}/\binom{t-h}{k-h} = \binom{t}{h}/\binom{k}{h}$ missing edges $F \in \binom{S}{h} \setminus \Hc_\ell$, which combined with \eqref{eq:boundonth} completes the proof of part \ref{item:hlessk} of Lemma~\ref{lem:kim}.
	\end{proof}
	
	\begin{proof}[Proof of Theorem~\ref{thm:stabHelly}]
		Fix $\ell \in \Zb$ and assume that the largest edge of $\Hc_{\ell +3}$ in $S$ is of size at most $(1-\varepsilon)|S|$ for some $\varepsilon > 0$. Since $\Hch$ has Helly Number $h$, this implies $\omega_h(\Hc_{\ell + 2}|_S) \leq (1 - \varepsilon)|S|$. Part \ref{item:hlessk} of Lemma~\ref{lem:kim} yields $\left | \binom{S}{h} \setminus \Hc_\ell \right | \geq \binom{k}{h}^{-1}\binom{\varepsilon |S|/h}{h} \geq \delta\cdot \binom{|S|}{h}$ with some $\delta=\delta(\varepsilon, k,h)>0$. Thus, if $\beta(\alpha)\leq1-\varepsilon$, then $\alpha\leq 1-\delta$, proving Theorem~\ref{thm:stabHelly}.
	\end{proof}
	
	In order to prove Theorem~\ref{thm:Qholmsen}, we need the following technical lemma, which is an analogue of Lemma~3.2 from \cite{holmsen2020large} and can be proved using part \ref{item:heqk} of Lemma~\ref{lem:kim}.
	
	\begin{lemma}\label{lem:newcolorclass}
		Let $\Hch$ be a hypergraph chain over a base set $V$ with Colorful Helly Number $k$. Let $S \subset V$ be a finite subset with $|S| = n$ large enough. If for a $t \in \Zb$ and $c \in (0,1)$ the inequality $\omega_k(\Hc_{t + 1}|_S) \leq cn/2$ holds, then given any $i \in [k]$ and a family $\mathcal{F}_i \subset \binom{S}{i}$ with $|\mathcal{F}_i| \geq c\binom{n}{i}$ there exists another family $\mathcal{F}_{i-1} \subset \binom{S}{i-1}$ and an $M \in \binom{S}{k} \setminus \Hc_t$ such that $|\mathcal{F}_{i-1}| \geq \left(\frac{c}{12k^2}\right)^k\binom{n}{n-1}$ and $A \cup \{v\} \in \mathcal{F}_i$ for all $A \in \mathcal{F}_{i-1}$ and $v \in M$.
	\end{lemma}
	
	\begin{proof}
		For every $A \in \binom{S}{i-1}$ let $\Gamma_A = \{v \in S: (A \cup \{v\}) \in \mathcal{F}_i\}$ and let \[\mathcal{P} = \left\{ (A,M): A \in \binom{S}{i-1}, M \in \binom{\Gamma_A}{k} \setminus \Hc_t \right\}.\]
		
		We want to lower bound $|\mathcal{P}|$. By part \ref{item:heqk} of Lemma~\ref{lem:kim}, for a fixed $A \in \binom{S}{i-1}$ there are at least $\binom{\frac{1}{k}\left( |\Gamma_A| - (c/2)n\right)}{k}$ distinct $M \in \binom{\Gamma_A}{k} \setminus \Hc_t$ such that $(A,M) \in \mathcal{P}$. Jensen's inequality gives
		
		\begin{align*}
			|\mathcal{P}| & \geq \sum_{A \in \binom{S}{i-1} } \binom{\frac{1}{k}\left( |\Gamma_A| - (c/2)n\right)}{k} \\ & \geq \binom{n}{i-1} \binom{\binom{n}{i-1}^{-1}\frac{1}{k}\sum_{A \in \binom{S}{i-1} } \left( |\Gamma_A| - (c/2)n\right)}{k}.
		\end{align*}
		Since
		\begin{align*}
			\sum_{A \in \binom{S}{i-1} } |\Gamma_A| = i|\mathcal{F}_i| \geq ic\binom{n}{i} > (n-i)c\binom{n}{i-1},
		\end{align*}
		we get
		\[\sum_{A \in \binom{S}{i-1} }\left( |\Gamma_A| - (c/2)n\right)  > (n-i)c\binom{n}{i-1} - (c/2)n\binom{n}{i-1},\]
		and thus
		\[|\mathcal{P}| \geq \binom{n}{i-1}\binom{\frac{nc}{2k} - \frac{ci}{k}}{k}.\]
		
		If $n$ is large enough compared to $i$ and $k$, then \[|\mathcal{P}| \geq \left( \frac{c}{12k^2}\right)^k \binom{n}{i-1}\binom{n}{k}.\]
		
		Since there are $\binom{n}{k}$ possible $M \in \binom{S}{k}$, there is an $M$ with at least $\left( \frac{c}{12k^2}\right)^k \binom{n}{i-1}$ different $A \in \binom{S}{i-1}$ such that $(A,M) \in \mathcal{P}$. These $A$ will form $\mathcal{F}_{i-1}$.
	\end{proof}
	
	\begin{proof}[Proof of Theorem~\ref{thm:Qholmsen}]
		We are given $\alpha\in(0,1)$, and our goal is to find the corresponding $\beta\in(0,1)$ satisfying Definition~\ref{def:QFHChain}. Let $f(x) = \left( \frac{x}{12k^2}\right)^k$, $\alpha_0 = \alpha$, $\alpha_{i+1} = f(\alpha_i)$. We will show that $\beta = \alpha_{k-1}$ is a good choice. Fix $\ell\in\Zb$ and suppose for a contradiction that $\left |\Hc_\ell \cap \binom{S}{k} \right | \geq \alpha \binom{n}{k}$, but $\Hc_{\ell + k+1}$ has no edge of size at least $\beta n$ inside $S$. Since $\Hch$ has Colorful Helly Number $k$, it has Helly Number at most $k$, so $\Hc_{\ell + k+1}$ having no edge of size at least $\beta n$ implies $\omega_k (\Hc_{\ell + k}|_S) < \beta n$.
		
		Set $\mathcal{F}_k = \Hc_\ell \cap \binom{S}{k}$. Since $\Hch$ is a hypergraph chain, $\mathcal{F}_k \subset \Hc_{\ell +i}$ for all $i \geq 0$, in particular, $\mathcal{F}_k \subset \Hc_{\ell + k}$. We have $|\mathcal{F}_k| \geq \alpha \binom{n}{k}$ and $\omega_k(\Hc_{\ell + k}|_S) < \beta n \leq (\alpha/2)n$, so we can apply Lemma~\ref{lem:newcolorclass} with $t=\ell+k-1$ and $c = \alpha$ to obtain an $\mathcal{F}_{k-1} \subset \binom{S}{k-1}$ with $|\mathcal{F}_{k-1}| \geq \alpha_1 \binom{n}{k-1}$ and an $M_1 \in \binom{S}{k} \setminus \Hc_{\ell + k-1}$ such that $A \cup \{v\} \in \mathcal{F}_{k}$ for all $A \in \mathcal{F}_{k-1}$ and $v \in M_1$. Now, we have $|\mathcal{F}_{k-1}| \geq \alpha_1 \binom{n}{k-1}$ and $\omega_k(\Hc_{\ell + k-1}|_S) \leq \omega_k(\Hc_{\ell + k}|_S) < \beta n \leq (\alpha_1/2)n$ and we can apply Lemma~\ref{lem:newcolorclass} again, this time with $t=\ell+k-2$ and $c = \alpha_1$, to obtain an $\mathcal{F}_{k-2} \subset \binom{S}{k-2}$ with $|\mathcal{F}_{k-2}| \geq \alpha_2 \binom{n}{k-2}$ and an $M_2 \in \binom{S}{k} \setminus \Hc_{\ell + k-2}$ such that $(A \cup \{v\}) \in \mathcal{F}_{k-1}$ for all $A \in \mathcal{F}_{k-2}$ and $v \in M_2$. Note that $(A \cup \{v_1, v_2\}) \in \mathcal{F}_k = \Hc_\ell \cap \binom{S}{k}$ for all $A \in \mathcal{F}_{k-2}$, $v_1 \in M_1$, $v_2 \in M_2$.
		
		After repeating this process $k-1$ times, we obtain an $\mathcal{F}_1 \subset \binom{S}{1}$ with $|\mathcal{F}_1| \geq \alpha_{k-1}n = \beta n$ and $M_1, \ldots, M_{k-1} \in \binom{S}{k} \setminus \Hc_{\ell + 1}$ such that $A \cup \{v_1, \ldots, v_{k-1}\} \in \Hc_\ell \cap \binom{S}{k}$ for all $A \in \mathcal{F}_{1}$, $v_1 \in M_1, \ldots ,v_{k-1} \in M_{k-1}$. Since $\omega_k(\Hc_{\ell +1}|_S) < \beta n$,  there must be an $M_k \in \binom{V(\mathcal{F}_1)}{k} \setminus \Hc_{\ell + 1}$. But regarding $M_1, \ldots, M_k$ as color classes, $\Hch$ having Colorful Helly-number $k$ yields a contradiction, since $M_1 \colortimes \ldots \colortimes M_k \subset \Hc_\ell$, but there is no color class $M_i \in \Hc_{\ell + 1}$.
	\end{proof}
	
	\section{Consequences for Quantitative Theorems}\label{sec:geometric}
	
	If $v = d^{-cd^2}$ from Theorem~\ref{thm:QCH}, then $\Quantchain$ has Colorful Helly Number $3d$ by Theorem~\ref{thm:QCH}, so the following Corollary follows from Theorem~\ref{thm:Qholmsen}.
	
	\begin{cor}\label{cor:QFH3d}
		For every dimension $d \geq 1$ and every $\alpha\in(0,1)$, there is a 
		$\beta\in(0,1)$ such that the following holds.
		
		\noindent Let $\C$ be a finite family of convex sets in $\Red$. Assume that among all 
		subfamilies of size $3d$, there are at least $\alpha \binom{|\C|}{3d}$ for 
		whom the intersection of the $3d$ members is of volume at least one.
		
		\noindent Then, there is a subfamily $\C^\prime\subset\C$ of size at least 
		$\beta |\C|$ such that $\vol{\bigcap\limits_{C\in \C^\prime}C}\geq d^{-cd^3}$ with a universal constant $c>0$.
	\end{cor}
	
	\begin{proof}
		The above claim is equivalent to saying that $\Quantchain$ has Fractional Helly Number $3d$, if $v = d^{-c'd^3}$ with a universal constant $c'$. Theorem~\ref{thm:QCH} states that $\Quantchain$ has Colorful Helly Number $3d$, if $v = d^{-cd^2}$ as in Theorem~\ref{thm:QCH}. By applying Theorem~\ref{thm:Qholmsen} to the latter Hypergraph Chain, we can conclude, that $\left(\mathbb{Q}_d\left(v^{(3d+1)\ell}\right) \right)_{\ell \in \Zb}$ has Fractional Helly Number $3d$ and $v = d^{-cd^2}$. But this is equvivalent to $\Quantchain$ having Fractional Helly Number $3d$ if $v = d^{-c'd^3}$.
	\end{proof}
	
	This is a slight improvement on the Fractional Helly Number, which was $3d+1$ in Theorem~\ref{thm:QFH}. Can we go below $3d$? Theorem~\ref{thm:stabHelly} implies at least a stability version of the Quantitative Helly Theorem with Helly Number $2d$ as follows.
	
	\begin{cor}\label{cor:QstabHelly}
		For every positive integer $d$ there exists a function $\beta: (0,1) \to [0,1)$ with $\lim_{\alpha \to 1} \beta ( \alpha) = 1 $ such that the following holds.
		
		\noindent Let $\C$ be a finite family of convex sets in $\Red$. Assume that among all 
		subfamilies of size $2d$, there are at least $\alpha \binom{|\C|}{2d}$ for 
		whom the intersection of the $2d$ members is of volume at least one.
		
		\noindent Then, there is a subfamily $\C^\prime\subset\C$ of size at least 
		$\beta |\C|$ such that $\vol{\bigcap\limits_{C\in \C^\prime}C}\geq d^{-cd^2}$ with a universal constant $c>0$.
	\end{cor}
	
	\begin{proof}
		Since $\Quantchain$, with $v = d^{-cd^2}$ from Theorem~\ref{thm:QCH}, has Helly Number $2d$ by Theorem~\ref{thm:QH} and Colorful Helly Number $3d$ by Theorem~\ref{thm:QCH}, we can apply Theorem~\ref{thm:stabHelly}. The assumption of Corollary~\ref{cor:QstabHelly} states that for a finite subset of convex sets $\C$, the inequality $\left|\mathcal{Q}_d(v^0) \cap \binom{\C}{2d}\right| \geq \alpha \binom{|\C|}{2d}$ holds with some $\alpha \in (0,1)$, where $v$ can be $v = d^{-cd^2}$ from Theorem~\ref{thm:QCH}. Theorem~\ref{thm:stabHelly} yields a subfamily $\C^\prime \subset \C$ with $\C^\prime \in \mathcal{Q}_d(v^3)$ and $|\C^\prime| \geq \beta (\alpha)|\C|$, where $\beta$ is the function from Theorem~\ref{thm:stabHelly}. For $\C^\prime$, the inequality $\vol{\bigcap\limits_{C\in \C^\prime}C}\geq \left(d^{-cd^2}\right)^3 = d^{-3cd^2}$ holds.
	\end{proof}

	\section{Remarks}\label{sec:rem}
	
	The following questions are left open.
	
	\begin{conj}\label{conj:QFH2d}
		For every dimension $d$, there is a $v = v(d) \in (0,1)$, such that $\Quantchain$ has Fractional Helly Number $2d$.
	\end{conj}
	
	\begin{conj}\label{conj:QCH2d}
		For every dimension $d$, there is a $v = v(d) \in (0,1)$, such that $\Quantchain$ has Colorful Helly Number $2d$.
	\end{conj}
	
	Our Theorem~\ref{thm:Qholmsen} shows that proving Conjecture~\ref{conj:QCH2d} would also confirm Conjecture~\ref{conj:QFH2d}.
	
	\section*{Acknowledgement}
	
	The author would like to thank Márton Naszódi for providing the problem and for all the valuable discussions during the research.
	
	\small
	\bibliographystyle{plain}
	\bibliography{biblio}

\end{document}